\documentclass[a4paper,10pt]{scrartcl}
\usepackage[utf8]{inputenc}
\usepackage[english]{babel}
\usepackage[fixlanguage]{babelbib}
\usepackage{cite}
\usepackage{amssymb, amsmath, amstext, amsopn, amsthm, amscd, amsxtra, amsfonts, commath}
\usepackage{yfonts, mathrsfs}
\usepackage{colonequals}
\usepackage{enumerate}
\usepackage{graphicx}
\usepackage[all]{xy}
\usepackage{todonotes}
\usepackage{tikz-cd}
\usepackage{hyperref}%
\hypersetup{
urlcolor = red,
colorlinks = true,
linkcolor = blue,
citecolor = blue,
linktocpage = true,
pdftitle = {Unit component of the NU Group},
pdfauthor = {Alexander Schmeding},
bookmarksopen = true,
bookmarksopenlevel = 1,
unicode = true,
hypertexnames =false,
}%
\usepackage{cleveref}
\swapnumbers
\newtheoremstyle{standard}
{16pt} 
{16pt} 
{} 
{} 
{\bfseries}
{} 
{ } 
{{\thmname{#1~}}{\thmnumber{#2.}}\thmnote{~(#3)}} 
\newtheoremstyle{kursiv}
{16pt} 
{16pt} 
{\itshape} 
{} 
{\bfseries}
{} 
{ } 
{{\thmname{#1~}}{\thmnumber{#2.}}\thmnote{~(#3)}} 
\theoremstyle{standard}
\newtheorem{defn}{Definition}[section]

\newtheorem{rem}[defn]{Remark}

\theoremstyle{kursiv}

\newtheorem{prop}[defn]{Proposition}
\newtheorem{cor}[defn]{Corollary}
\newtheorem{lem}[defn]{Lemma}

\newcommand{\R}{\ensuremath{\mathbb{R}}}

\newcommand{\N}{\ensuremath{\mathbb{N}}}

\newcommand{\SSS}{\ensuremath{\mathbb{S}}}

\newcommand{\coloneq}{\colonequals}
\DeclareMathOperator{\Diff}{Diff}

\DeclareMathOperator{\id}{id}

\DeclareMathOperator{\Lor}{\mathrm{SO}^+ (3,1)}

\DeclareMathOperator{\NU}{NU}

\newcommand{\z}{\mathbf{z}}

\begin{document}
\title{On the unit component of the Newman-Unti group}
\author{Alexander Schmeding\footnote{Nord universitetet i Levanger, Norway: \href{mailto:alexander.schmeding@nord.no}{alexander.schmeding@nord.no}}}
\date{}
{\let\newpage\relax\maketitle}

\begin{abstract}
	In this short note we identify the unit component of the Newman--Unti (NU) group in the fine very strong topology. In previous work, this component has been endowed with an infinite-dimensional Lie group structure, while the full NU-group does not support such a structure. 
\end{abstract}


\medskip

\\[1em]

In \cite{Prinz_Schmeding_1, Prinz_Schmeding_2}, the Lie group structure of the Bondi--Metzner--Sachs (BMS) group and the Newman--Unti (NU) group as well as certain related groups were studied from the viewpoint of infinite-dimensional geometry. Both groups appear in the theory of asymptotically flat space times in the sense of Bondi, Van der Burg, Metzner and Sachs. On the infinitesimal level, the Lie algebra of the more widely studied BMS group identifies as s subalgebra of the Lie algebra of the NU group (see \cite{Barnich:2009se,Barnich:2010eb}). The global level is more delicate, as only the unit component of the NU group can be turned into a Lie group in a suitable function space topology. It is thus of interest to identify this component of the NU group. To this end, recall that the NU group is a semidirect product of a certain set of smooth functions with the orthochronous Lorentz group
\begin{equation}\label{eq:NUsemidirect}
\NU = \mathcal{N} \rtimes \Lor  \, \text{ with } \mathcal{N} \subseteq C^\infty (\R \times \SSS^2).
\end{equation}
The set of functions $\mathcal{N}$ can be endowed with a variety of function space topologies (derived from an ambient space of smooth functions). We have shown in \cite{Prinz_Schmeding_2}, that the usual function space topologies do not turn the NU group into an (infinite-dimensional) Lie group. However, the unit component of the NU group can be turned into an infinite-dimensional Lie group with respect to the so called fine very strong topology. 

In loc.cit. we then mentioned en passant that the unit component of the NU group coincides with the subgroup $\NU_c = \mathcal{N}_c \times \Lor$ of $\NU$, where
$$\mathcal{N}_c \coloneq \{F \in \mathcal{N} \mid \exists K \subseteq \R \times \SSS^2 \text{ compact, s.t.} F(t,\z) = t \ \forall (t,\z) \not \in K \}.$$
While this is not hard to prove the argument is somewhat involved and it is thus the goal of this note to supply the argument.

\textbf{Acknowledgements:} The author wishes to thank H.~Gl\"{o}ckner for the question which led to this note. Moreover, he thanks D.~Prinz and H.~Gl\"{o}ckner for helpful comments on the draft.

\section{Identifying the unit component}

Let us note first that the problem reduces to a question about the function space $\mathcal{N}$. As a semidirect product of Lie groups is, topologically, a product of the underlying topological spaces, the unit component is given by the product of the connected components containing the parts of the unit in the product decomposition. Now as the orthochronous Lorentz group is connected, the semidirect product structure becomes irrelevant and it suffices to identify the connected component containing the unit element in the subgroup $\mathcal{N}$. The topology at hand is here given by the subspace topology induced by the fine very strong topology. What we will now set out to prove is that this component coincides with all mappings in $\mathcal{N}$ differing from the unit element only on a compact set. For the readers convenience we shall now first recall the definitions of the sets involved.

Let $\Diff^+ (\R)$ be the group of all smooth orientation-preserving diffeomorphisms of $\R$ (recall that a diffeomorphism of $\R$ is orientation-preserving if it has positive derivative everywhere). We note that the orientation preserving diffeomorphisms of the real line admit a global manifold chart (see \cite{Gl05} for details concerning the construction). It is given by the map
$$\kappa \colon \Diff^+ (\R) \rightarrow C^\infty (\R,\R),\quad \kappa(\phi) = \phi-\id_\R,$$
which identifies the diffeomorphism with the open convex set $$C\coloneq \{f \in C^\infty (\R,\R) \mid f'(t)>-1, \forall t \in \R\}.$$
As translation with a fixed function induces a homeomorphism of $C^\infty (\R,\R)$ in the fine very strong topology (cf.\ \cite{HaS17}), this implies that $\Diff^+ (\R)$ is an open convex subset of $C^\infty (\R,\R)$.

\begin{defn} \label{defn:mathcalN}
 Define $\mathcal{N} \coloneq \left\{F \in C^\infty (\R \times \SSS^2) \mid F(\cdot, \z) \in \Diff^+(\R) , \quad \forall \z \in \SSS^2\right\}$. Then $\mathcal{N}$ becomes a group with respect to the product
 $$F \cdot G (u,\z) \coloneq F(G(u,\z),\z).$$
 The unit of the product is the map $p \colon \R \times \SSS^2 \rightarrow \R, (t,\z) \mapsto t$ and the group inverse $F^{-1}\colon \R \times \SSS^2\rightarrow \R$ is for $\z \in \SSS^2$ given by $F(\cdot, \z)^{-1}$. Here for every $\z$, the inverse is the unique smooth map satisfying the implicit equation 
 \begin{align}\label{NU:imp:eq}
  t = F(F^{-1}(t,\z),\z) \qquad (t,\z) \in \R \times \SSS^2.
 \end{align}
\end{defn}

We now topologise $\mathcal{N}$ with the subspace topology induced by the fine very strong topology on $C^\infty (\R \times \SSS^2,\R)$ (we do not recall its definition here, but see \cite{Prinz_Schmeding_2} or \cite{HaS17}.
Note that the fine very strong topology turns the set 
$$\mathcal{N}_c \coloneq \{F \in \mathcal{N} \mid \exists K \subseteq \R \times \SSS^2 \text{ compact, s.t.} F(t,\z) = t \ \forall (t,\z) \not \in K\}$$
into an open subset which contains the unit $p$. Recall from the definition of the fine very strong topology (see e.g.\ \cite{HaS17}) that the connected component of the unit in $\mathcal{N}$ contains only functions which differ at most on a compact subset from the unit $p$. Hence, the connected component $\mathcal{N}^\text{fs}_0$ of the unit is contained in $\mathcal{N}_c$. Therefore, it suffices to prove that $\mathcal{N}_c$ is (path)connected. To establish the claimed connectedness, we need the following lemma, whose proof can be found in \cite[Lemma C.3]{AaS19}:\footnote{The cited lemma is stated in loc.cit. only for manifolds without boundary. Inspecting the proof, it is easy to see that it carries over without any changes to the case of a manifold with smooth boundary (or even more general to manifolds with rough boundary).}

\begin{lem}\label{aux1}
 Let $M$ be a manifold (possibly infinite-dimensional or with smooth boundary) and $X$ a smooth locally compact manifold. Let $f \colon M \times X \rightarrow \R^n, n \in \N$ be smooth such that $f$ vanishes outside $M \times K$ for $K \subseteq X$ compact. Then the adjoint map $f^\wedge \colon M \rightarrow C^\infty_c (X,\R^n), m \mapsto f(m,\cdot)$ is smooth, where the target space $C^\infty_c(X,\R^n) \coloneq \{f \in C^\infty (X,\R^n) \mid \exists L \subseteq X \text{ compact s.t. } f|_{X\setminus L}\equiv 0\}$ is an open subset of $C^\infty (X,\R^n)$ in the fine very strong topology.
\end{lem}

\begin{prop}\label{prop:pathcon}
 The set $\mathcal{N}_c$ is (path-)connected in $\mathcal{N}$ with respect to the fine very strong topology.
\end{prop}

\begin{proof} 
Consider $F_0,F_1 \in \mathcal{N}_c$ and let us construct a continuous path from $F_0$ to $F_1$. For this, we study the mapping 
$$
c \colon [0,1] \times (\R \times \SSS^2) \rightarrow \R, \quad (s,(t,\z))\mapsto sF_1 (t,\z) + (1-s)F_0(t,\z)-t. 
$$
By construction $c$ is smooth. Moreover, we observe that since $F_0,F_1 \in \mathcal{N}_c$, there exist compact sets $K_i, i=0,1$ such that $F_i(t,\z)=t$ for all $(t,\z)\not \in K_i$. In particular, for every $(t,\z) \not \in L\coloneq K_1 \cup K_2$ we have $c(s,t,\z) = 0$. Thus we invoke \Cref{aux1} to obtain a smooth (and thus continuous) path
$$c^\wedge \colon [0,1] \rightarrow C^\infty_c (\R\times \SSS^2,\R),\quad t \mapsto c(t,\cdot).$$
Addition of a function induces a homeomorphism of $C^\infty(\R\times \SSS^2,\R)$ to itself when endowed with the fine very strong topology (note that $C^\infty(\R\times \SSS^2,\R)$ is disconnected in this topology whence not a locally convex vector space). We deduce that $\gamma \colon [0,1] \rightarrow C^\infty (\R\times \SSS^2,\R), s \mapsto c^\wedge (s,\cdot) + p$ (with $p$ being the unit of $\mathcal{N}$) is a continuous path. 

We have to prove now that $\gamma$ takes its image in $\mathcal{N}_c$. For this it is sufficient to prove that $\gamma$ takes its values in $\mathcal{N}$. By construction for every fixed $\z \in \SSS^2$ we obtain the map $c (\cdot,\z)^\wedge \colon [0,1] \rightarrow C^\infty_c (\R, \R), s \mapsto c(s,\cdot,\z)$ is continuous in the parameter $s$ and a convex combination of the diffeomorphisms $F_0$ and $F_1$. Since $\Diff^+(\R)$ is convex, we see that $c(t,\cdot ,\z) \in \Diff^+(\R)$ as $F_0,F_1$ are contained in this set. Thus \Cref{defn:mathcalN} implies that $\gamma(t) \in \mathcal{N}$ and thus $\mathcal{N}_c$ is path connected.
 \end{proof}

 \begin{cor}
  The unit component of the group $\mathcal{N}$ coincides with $\mathcal{N}_c$.
 \end{cor}

\begin{rem}
 In \cite{Prinz_Schmeding_2} $\mathcal{N}_c$ was turned into a Lie group, these techniques generalise (without essential changes) also to similar groups constructed from mappings in $C^\infty (\R^n \times \SSS^2,\R^n)$ (the author is not aware of any physical relevance of these groups). For these generalised groups the identification of the unit component presented in this note does not work as the diffeomorphism group $\Diff^+(\R^n) $ is not known to be convex (beyond the special case $n=1$).
\end{rem}

\addcontentsline{toc}{section}{References}
\bibliography{unit_rem}
\end{document}